\documentclass{amsart}
\usepackage{MnSymbol}
\usepackage[mathscr]{euscript}
\usepackage{scalerel}
\usepackage{colonequals}
\usepackage{xcolor}

\usepackage{hyperref}
\hypersetup{
    colorlinks,
    citecolor=black,
    filecolor=black,
    linkcolor=black,
    urlcolor=black
}
\usepackage{comment}

\DeclareFontFamily{OT1}{pzc}{}
\DeclareFontShape{OT1}{pzc}{m}{it}{<-> s * [1.100] pzcmi7t}{}
\DeclareMathAlphabet{\mathpzc}{OT1}{pzc}{m}{it}

    \usepackage{etoolbox}
    \patchcmd{\section}{\scshape}{\large\bfseries}{}{}
    \makeatletter
    \renewcommand{\@secnumfont}{\bfseries}
    \makeatother

\usepackage{tikz-cd}
\tikzcdset{arrow style=math font}
\tikzcdset{scale cd/.style={every label/.append style={scale=#1},
    cells={nodes={scale=#1}}}}
\usepackage{tikz}
\usetikzlibrary{arrows}
\usepackage{caption}
\usepackage{float}

\usepackage[normalem]{ulem}

\usepackage[maxbibnames=99]{biblatex}
\usepackage{booktabs}
\bibliography{references}

\numberwithin{equation}{section}
\newtheorem{theorem}{Theorem}[section]
\newtheorem*{theorem*}{Theorem}
\newtheorem{corollary}[theorem]{Corollary}
\newtheorem{lemma}[theorem]{Lemma}
\newtheorem{proposition}[theorem]{Proposition}
\theoremstyle{definition}

\newtheorem*{question*}{Question}
\newtheorem*{conjecture*}{Conjecture}

\newtheorem{definition}[theorem]{Definition}
\newtheorem{remark}[theorem]{Remark}
\newtheorem{example}[theorem]{Example}

\definecolor{mycolor1}{rgb}{1,0,0}
\definecolor{mycolor2}{rgb}{0,0.8,0}
\definecolor{mycolor3}{rgb}{0,0,0.9}

\def\Ker{\mathrm{Ker}}

\def\KK{\mathbb{K}}
\def\ZZ{\mathbb{Z}}
\def\MH{\mathrm{MH}}
\def\NN{\mathcal{N}}

\def\FF{\mathcal{F}}
\def\GG{\mathcal{G}}
\def\Hom{\mathrm{Hom}}
\def\Tor{\mathrm{Tor}}

\def\PH{\mathrm{PH}}
\def\RR{\mathcal{R}}
\def\AA{\mathcal{A}}
\def\JJ{\mathcal{J}}
\def\FF{\mathbb{F}}
\def\QQ{\mathbb{Q}}
\def\SS{\mathcal{S}}

\usetikzlibrary{shapes.geometric, calc}
\usetikzlibrary{decorations.markings,shapes.arrows, shapes.symbols}
\usetikzlibrary{decorations.pathreplacing}

\tikzset{
  on each segment/.style={
    decorate,
    decoration={
      show path construction,
      moveto code={},
      lineto code={
        \path [#1]
        (\tikzinputsegmentfirst) -- (\tikzinputsegmentlast);
      },
      curveto code={
        \path [#1] (\tikzinputsegmentfirst)
        .. controls
        (\tikzinputsegmentsupporta) and (\tikzinputsegmentsupportb)
        ..
        (\tikzinputsegmentlast);
      },
      closepath code={
        \path [#1]
        (\tikzinputsegmentfirst) -- (\tikzinputsegmentlast);
      },
    },
  },
  mid arrow/.style={postaction={decorate,decoration={
        markings,
        mark=at position .75 with {\arrow[#1]{stealth}}
      }}},
}

\let\oldtocsection=\tocsection 
\let\oldtocsubsection=\tocsubsection 
\renewcommand{\tocsection}[2]{\hspace{0mm}\oldtocsection{#1}{#2}}
\renewcommand{\tocsubsection}[2]{\hspace{1em}\oldtocsubsection{#1}{#2}}

\title[Path homology of digraphs without multisquares]{Path homology of digraphs without multisquares and its comparison with homology of spaces}

\author{Xin Fu} 
\address{
Beijing Institute of Mathematical Sciences and Applications (BIMSA)}
\email{x.fu@bimsa.cn}
\thanks{X.\,Fu is supported by the Beijing Natural Science Foundation (grant no.\,1244043).}

\author{Sergei O. Ivanov} 
\address{
Beijing Institute of Mathematical Sciences and Applications (BIMSA)}
\email{ivanov.s.o.1986@gmail.com, ivanov.s.o.1986@bimsa.cn}

\begin{document}

\begin{abstract}
For a digraph $G$ without multisquares and a field $\FF$, we construct a basis of the vector space of path $n$-chains $\Omega_n(G;\FF)$ for $n\geq 0$, generalising the basis of $\Omega_3(G;\FF)$ constructed by Grigory'an. For a field $\FF,$ we consider the $\FF$-path Euler characteristic $\chi^\FF(G)$ of a digraph $G$ defined as the alternating sum of dimensions of path homology groups with coefficients in $\FF.$ If $\Omega_\bullet(G;\FF)$ is a  bounded chain complex, the constructed bases can be applied to compute $\chi^\FF(G)$. We provide an explicit example of a digraph $\GG$ whose $\FF$-path Euler characteristic depends on whether the characteristic of $\FF$ is two, revealing the differences between GLMY theory and the homology theory of spaces. This allows us to prove that there is no topological space $X$ whose homology is isomorphic to path homology of the digraph $H_*(X;\KK)\cong \PH_*(\GG;\KK)$ simultaneously for $\KK=\ZZ$ and $\KK=\ZZ/2\ZZ.$
\end{abstract}

\maketitle

\tableofcontents

\section{Introduction}

The path homology of digraphs, developed by Grigory’an-Lin-Muranov-Yau, and their coauthors in a series of papers \cite{grigor2012homologies,grigor2014homotopy, grigor2017homologies, grigor2020path, grigor2014graphs, grigor2018path, grigor2018path2}, has many parallels with the homology theory of spaces. 
These include an analogue of the K\"unneth formula for the box product of digraphs and  join of digraphs, a homotopy category of digraphs, and the path homology being a homotopy invariant.
Moreover, there is a notion of a fundamental group, whose abelianization coincides with the first path homology group over integers. Attempts have been made to build a consistent theory of homotopy groups \cite{li2024homotopy} and to generalize this theory to various other structures, such as quivers \cite{grigor2018path, ivanov2024simplicial}. This entire  research direction is now known as GLMY theory. As Asao shows in~\cite{asao2023magnitude}, the path homology is closely related to the magnitude homology. Based on this connection,  Hepworth and Roff developed the bigraded path homology of digraphs~\cite{hepworth2024bigraded}, which was also generalised to arbitrary quasimetric spaces in \cite{ivanov2023nested}.

As with spaces, for a digraph $G$ and a field $\FF$, one can construct a chain complex~$\Omega_\bullet(G;\FF)$ and a cochain complex~$\Omega^\bullet(G;\FF)$, whose homology and cohomology are called path homology $\PH_*(G;\FF)$ and path cohomology $\PH^*(G;\FF)$, respectively. 
Furthermore, $\Omega^\bullet(G;\FF)$ is a differential graded algebra (dg-algebra), and this structure turns the path cohomology into a graded algebra. However, unlike the case of spaces, the definitions of the complexes $\Omega_\bullet(G;\FF)$ and $\Omega^\bullet(G;\FF)$ are given in such a way that there is no explicit description of the basis of its components.

It is known that the dimension of $\Omega_0(G;\FF)$ equals the number of vertices, the dimension of $\Omega_1(G;\FF)$ equals the number of arrows, and the dimension of $\Omega_2(G;\FF)$ equals the number of triangles plus the number of squares. A digraph $G$ is said to have no multisquares if, for any pair of vertices $x,y$ with directed distance two, there are at most two shortest paths from $x$ to $y.$ Grigory’an showed \cite[Theorem~2.10]{grigor2022advances} that if a digraph $G$ does not have multisquares, a basis of $\Omega_3(G;\FF)$ consists of special digraphs called trapezohedrons and their mapping images.
However, for general cases, the dimensions of $\Omega_n(G;\FF)$ remain unclear.

In this paper, we describe a basis of~$\Omega_n(G;\FF)$ for~$G$ containing no multisquares and $n\geq 0$, which generalises the basis of $\Omega_3(G;\FF)$ by Grigory'an. We start by analysing the identification of the dg-algebra~$\Omega^\bullet(G;\FF)$ as a quotient of the path algebra of $G$, as described in~\cite[Theorem 5.4]{ivanov2024diagonal},
\[
\Omega^\bullet(G;\FF) \cong \FF G/T,
\]
where $T$ is the ideal defined by  certain relations $t_{x,y}$; see also Theorem \ref{theorem:omega:description}. 
 When~$G$ does not have multisquares, the relations defining this quotient have a simpler form. This allows us to describe a basis for $\Omega^n(G;\FF)$, yielding the dual basis for~$\Omega_n(G;\FF)$ (see Corollary~\ref{corollary:basis:omega^n} and Theorem~\ref{theorem:dual-basis}).

To proceed, we construct an undirected graph $\SS_n(G)$, whose vertices correspond to the $n$-paths in~$G$ (see Definition~\ref{def graph of short moves}). 
The dimension of $\Omega_n(G;\FF)$ 
can be calculated by counting the number of certain connected components of $\SS_n(G)$. Notably, it depends on the characteristic of $\FF$. For fields with characteristic two, the dimension of~$\Omega_n(G;\FF)$ may be greater than for fields of other characteristics. 
When the characteristic is not two, an additional bipartite property is required for the components of $\SS_n(G)$ corresponding to the elements of the basis, unlike the case when the characteristic is two.
Basis elements of~$\Omega_n(G;\FF)$ are expressed as a certain linear sum of vertices from each of these components of~$\SS_n(G)$. Refer to~Section~\ref{section basis} for a detailed description of the basis.

If the number of nontrivial path homology groups $\PH_n(G;\FF)$ of a digraph $G$ is finite, we define the $\FF$-path Euler characteristic of~$G$ as the alternating sum
\[
\chi^\FF(G) = \sum_n (-1)^n \dim_\FF(\PH_n(G;\FF)).
\]
When the complex $\Omega_\bullet(G;\FF)$ is bounded, this is equal to the alternating sum of the dimensions of $\Omega_n(G;\FF)$. 
Therefore, our description of bases for $\Omega_n(G;\FF)$ can be used to compute~$\chi^{\FF}(G)$ of~a digraph $G$ that does not contain multisquares and whose associated chain complex $\Omega_\bullet(G;\FF)$ is bounded.

Unlike spaces, the $\FF$-path Euler characteristics can vary depending on the field~$\FF$. This demonstrates the differences between GLMY theory and the homology theory of spaces. To illustrate this, we provide an example of a digraph~$\mathcal{G}$ in Subsection~\ref{sec main example}. This example highlights the fact that for certain digraphs, it is not possible to construct a space with homology isomorphic to the path homology for all coefficients. 
As a side note, we discover that over integers, the associated path cochain complex $\Omega^\bullet(\GG;\ZZ)$  contains torsion elements.

To summarise, we prove the following theorem (see Theorem \ref{theorem:example} and Proposition~\ref{prop:general:non-existance}).

\begin{theorem*}
For the digraph $\GG$ defined in~\eqref{eq main example}, the following statements hold:
\begin{enumerate}
\item The rational and mod-$2$ path Euler characteristics of $\GG$ are not equal 
\[ \chi^\QQ(\GG) \neq \chi^{\ZZ/2\ZZ}(\GG).\]

\item There is no space $X$ such that there is an isomorphism of graded abelian groups $\PH_*(\GG;\KK)\cong H_*(X;\KK)$ simultaneously for $\KK=\ZZ$ and $\KK=\ZZ/2.$

\item There is no space $X$ with homology of finite type such that there is an isomorphism of graded vector spaces $\PH_*(\GG;\FF)\cong H_*(X;\FF)$ simultaneously for $\FF=\QQ$ and $\FF=\ZZ/2\ZZ$.

\item There is no chain complex of free abelian groups $C_\bullet$ with homology of finite type such that $\PH_*(\GG;\FF) \cong H_*(C_\bullet \otimes_\ZZ \FF)$ simultaneously for  $\FF=\QQ$ and  $\FF=\ZZ/2\ZZ.$
\item There exists $n$ such that there is no ``universal coefficient'' short exact sequence 
\[ 0 \to \PH_n(\GG,\ZZ)\otimes \ZZ/2\ZZ \to    \PH_n(\GG;\ZZ/2\ZZ) \to  \Tor(\PH_{n-1}(\GG,\ZZ),\ZZ/2\ZZ) \to 0.\]

\item $\Omega^4(\GG;\ZZ)$ has $2$-torsion.
\end{enumerate} 
\end{theorem*}

{\bf Open questions.} In conclusion of the introduction, we present a list of open questions that continue this line of research.

\begin{itemize}
\item For a digraph $G$ and a field $\FF,$ is the path cohomology algebra $\PH^*(G;\FF)$ graded commutative?

\item Let $G$ be a digraph and $p$ be an odd prime such that $\Omega_\bullet(G;\QQ)$ and $\Omega_\bullet(G;\ZZ/p\ZZ)$ are bounded chain complexes. Our results imply that, if $G$ has no multisquares, we have 
\[\chi^{\QQ}(G) = \chi^{\ZZ/p\ZZ}(G).\]
Does this property also hold for  digraphs with multisquares?

\item  For a digraph $G$ without multisquares, are all nontrivial entries of the matrix representing the boundary map $\partial: \Omega_n(G;\QQ)\to \Omega_{n-1}(G;\QQ)$ with respect to the constructed bases equal to $\pm 1$?
\end{itemize}

\section{Path chains and cochains}

Throughout the article, all digraphs are finite, $\KK$ denotes a commutative ring, and  $\FF$ denotes a filed. Our terminology differs from that used in~\cite{grigor2012homologies} and other works of Grigor'yan-Lin-Muranov-Yau, but our notations are the same. For example, by a \emph{path} in $G$, we mean a sequence of vertices $p=(p_0,\dots,p_n)$ such that~$(p_i,p_{i+1})$ is an arrow for $0\leq i\leq n-1$. If we need to specify the length $n,$ we call it an $n$-path, or a path of length $n.$ Elements of~$\Omega_\bullet(G)$ are called path chains, and elements of~$\Omega^\bullet(G)$ are called path cochains. 
We also use asterisks $H_*$, $H^*$ for graded objects (without a differential) and bullet points $C_\bullet$, $C^\bullet$ for (co)chain complexes and dg-algebras.

\subsection{Definition}

We start with the definitions of $\Omega_\bullet(G)$ and $\Omega^\bullet(G)$ following Grigor'yan-Lin-Muranov-Yau~\cite{grigor2012homologies}; see also \cite[\S 5]{ivanov2024diagonal}. We will work over arbitrary commutative ring~$\KK$ rather than a field. Later, when discussing duality between $\Omega_\bullet(G)$ and $\Omega^\bullet(G),$ we will assume that $\KK=\FF$ is a field.

Assume that $\KK$ is a commutative ring, and by a ``module'' we mean a $\KK$-module. For a finite set $X$, denote by $\Lambda_\bullet(X)$ the chain complex of modules, where the $n$-th component $\Lambda_n(X)$ is a free module generated by formal elements $e_{x_0,\dots,x_n}$ indexed by tuples $(x_0,\dots,x_n)\in X^{n+1}$. In other words, we have
\begin{equation}\label{eq basis}
e_{x_0,\dots,x_n} \in \Lambda_n(X).
\end{equation}
The differential on $\Lambda_\bullet(X)$ is defined by the alternating sum 
\begin{equation}
\partial(e_{x_0,\dots,x_n})=\sum_{i=0}^n (-1)^n e_{x_0,\dots, \hat x_i, \dots x_n}. \end{equation}
Denote by $\RR_\bullet(X)$ the quotient complex of $\Lambda_\bullet(X)$ by the subcomplex whose components are spanned by tuples $e_{x_0,\dots,x_n}$ with consecutive repetitions, i.e., $x_i=x_{i+1}$ for some $0\leq i\leq n-1.$ 
A reader familiar with simplicial methods may notice that~$\Lambda_\bullet(X)$ has an obvious structure of a simplicial module, and $\RR_\bullet(X)$ is the associated normalised complex. 

For a digraph $G$, we set $\RR_\bullet(G)=\RR_\bullet(V(G)),$ where~$V(G)$ is the set of vertices. We denote by $\AA_n(G)\subseteq \RR_n(G)$ the submodule generated by all paths of length $n.$ So $\AA_n(G)$ consists of elements that can be written as linear combinations 
\begin{equation}
\sum_p \alpha_p\cdot e_p \in \AA_n(G),
\end{equation}
where $p$ runs over all $n$-paths and $\alpha_p\in \KK.$
Then $\Omega_\bullet(G)$ is defined as the largest subcomplex of $\RR_\bullet(G)$ that is contained in $\AA_*(G).$ 
The $n$-th component $\Omega_n(G)$ can be also defined by the following formula
\begin{equation}\label{eq defn Omega n}
\Omega_n(G) = \AA_n(G)\cap \partial^{-1}(\AA_{n-1}(G)).
\end{equation}
We call $\Omega_\bullet(G)$ the path chain complex and call an element in $\Omega_\bullet(G)$ a path chain.  If we need to specify the ring $\KK$, we write $\Omega_\bullet(G; \KK).$

Dually, the cochain complex~$\Lambda^\bullet(X)$ is given by
\[
\Lambda^\bullet(X)=\Hom_{\KK}(\Lambda_\bullet(X),\KK).
\]
The elements $e^{x_0,\dots,x_n}\in \Lambda^n(X)$ denote the dual basis of $e_{x_0,\dots,x_n}$ in $\Lambda_n(X)$  indexed by tuples~$(x_0,\dots,x_n)$. Then the differential on $\Lambda^\bullet(X)$ follows by the formula
\begin{equation}
\partial(e^{x_0,\dots,x_n}) = \sum_{i=0}^{n+1} \sum_{v\in X} (-1)^i e^{x_0,\dots,x_{i-1},v,x_{i},\dots,x_n}.
\end{equation}
We turn $\Lambda^\bullet(X)$ into a dg-algebra by letting
\begin{equation}
e^{x_0,\dots,x_n} \cdot e^{y_0,\dots,y_m} = \begin{cases}
 e^{x_0,\dots,x_n,y_1,\dots,y_m}, & y_0=x_n\\
 0, & \text{else}.
\end{cases}
\end{equation}

Consider the dg-subalgebra  $\RR^\bullet(X)\subseteq \Lambda^\bullet(X),$ where $\RR^n(X)$ is spanned by the elements $e^{x_0,\dots,x_n}$ such that $x_i\neq x_{i+1}$ for all $0\leq i\leq n-1$. 
A reader familiar with simplicial methods may note that $\Lambda^\bullet(X)$ has an obvious structure of a cosimplicial algebra, where the codegeneracy map $\sigma^i:\Lambda^n(X)\to \Lambda^{n-1}(X)$ is trivial on $e^{x_0,\dots,x_n}$ if~$x_i\neq x_{i+1}$ and $\sigma^i(e^{x_0,\dots,x_n})=e^{x_0,\dots,\hat x_i,\dots,x_n}$ if~$x_i=x_{i+1}.$  
Then $\RR^\bullet(X)$ is the associated normalised dg-algebra. 

For a digraph $G$, let us define the dg-algebra $\Omega^\bullet(G)$ of path cochains. 
First we set $\RR^\bullet(G)=\RR^\bullet(V(G)).$ 
Denote by $\NN^n(G)\subseteq \RR^n(G)$ the subspace spanned by elements $e^{x_0,\dots,x_n},$ where $(x_0,\dots,x_n)$ is not a path of $G.$ Then $\NN^*(G)$ is a graded ideal of $\RR^\bullet(G)$ but not necessarily a dg-ideal. Let $\JJ^\bullet(G)\triangleleft \RR^\bullet(G)$ be the smallest dg-ideal containing $\NN^*(G)$. 
Then the $n$-th component $\JJ^n(G)$ can be expressed as 
\begin{equation}
\JJ^n(G) = \NN^n(G) + \partial(\NN^{n-1}(G)).
\end{equation}
The dg-algebra of path cochains is defined as the quotient algebra 
\begin{equation}
\Omega^\bullet(G) = \RR^\bullet(G)/\JJ^\bullet(G).
\end{equation}

\subsection{Path cochain algebra as a quadratic algebra}

Denote by $\KK G$ the path algebra of $G.$ As a module, it is spanned by paths $(x_0,\dots,x_n)$ for $n\geq 0,$ including paths of length zero. 
The product of two paths is defined by 
\begin{equation}
(x_0,\dots,x_n)(y_0,\dots,y_m) 
=
\begin{cases}
(x_0,\dots,x_n,y_1,\dots,y_m), &\text{ if } x_n=y_0;\\
0,& \text{ else.}
\end{cases}
\end{equation}
The unit $1\in \KK G$ is the sum of paths of length zero  \[1\colonequals \sum_{v\in V(G)} (v).\]

Note that $\KK G$ is a graded algebra, where the  grading of a path $(x_0,\dots,x_n)$ equals its length, written as 
\[|(x_0,\dots,x_n)|=n.\] 
Let us write $(\KK G)^n$ as the $n$-th component of $\KK G$.

For any vertices $x,y\in V(G)$ such that $d(x,y)=2,$ consider the element $t_{x,y}$ defined as the linear sum of all $2$-paths $(x,v,y)$ from $x$ to $y$:
\begin{equation}
t_{x,y} = \sum_{v\in V(G)} (x,v,y) \in (\KK G)^2.
\end{equation}
We denote by $T\triangleleft \KK G$ the ideal generated by all such elements $t_{x,y}.$  

\begin{lemma}\label{lemma isomorphism} \ 
\begin{enumerate}
\item There is an isomorphism of graded algebras
\begin{equation}\label{eq:iso-R/N-KG}
\varphi : \RR^\bullet(G)/\NN^\bullet(G) \xrightarrow{\cong} \KK G
\end{equation}
sending $e^{x_0,\dots,x_n}+\NN^\bullet(G)$ to $(x_0,\dots,x_n)$, for any path $(x_0,\dots,x_n)$ in $G$. 
\smallskip
\item The image of $\JJ^\bullet(G)/\NN^\bullet(G)$ under $\varphi$ is the ideal $T$.
\end{enumerate}
\end{lemma}

\begin{proof}
    See \cite[Lemmas 5.2 \& 5.3]{ivanov2024diagonal}.
\end{proof}

\begin{theorem}[{\cite[Th. 5.4]{ivanov2024diagonal}}]
\label{theorem:omega:description}
The isomorphism \eqref{eq:iso-R/N-KG} induces an isomorphism of dg-algebras
\[
\Omega^\bullet(G) \cong \KK G/T,
\]
where the differential on $\KK G/T$ is defined  by the formula  
\[
\partial (p+T) = \sum_{i=0}^{n+1}\  \sum_{v} (-1)^i (p_0,\dots,p_{i-1},v,p_{i},\dots,p_n)+T,
\]
for any path $p=(p_0,p_1,\ldots, p_n)$ in $G$,
and $v$ runs over all vertices such that \[(p_0,\dots,p_{i-1},v,p_{i},\dots,p_n)\] is a path in $G$.
\end{theorem}

\subsection{Duality between path chains and cochains}

In this subsection, we assume that $\KK=\FF$ is a field. The isomorphism
\[
\RR^\bullet(G)\cong \Hom_\FF( \RR_\bullet(G),\FF), \quad
e^{x_0,\dots,x_n}\mapsto e_{x_0,\dots,x_n}
\]
induces an isomorphism by restricting to $\AA_\bullet(G)$:
\begin{equation}\label{eq isomorphism dual of A(G)}
\RR^\bullet(G)/\NN^\bullet(G)\cong \Hom_\FF( \AA_\bullet(G),\FF).
\end{equation}
Over $\FF$, the above isomorphism~\eqref{eq isomorphism dual of A(G)} induces an isomorphism 
\begin{equation}\label{eq isomorphism dual}
\Omega^\bullet(G)\cong \Hom_\FF (\Omega_\bullet(G),\FF).
\end{equation}
Hence $\Omega^\bullet(G)$ is dual to $\Omega_\bullet(G)$. 
See~\cite[Lemma 3.19]{grigor2012homologies} or~\cite[Prop. 3.16]{ivanov2024simplicial} for more details.

Given any $n$-paths $p$ and $q$ in $G$, the Kronecker delta $\delta_{p,q}$
determines a non-degenerate bilinear form
\begin{equation}\label{eq:bilinear form}
\langle - , - \rangle : \AA_n(G) \times (\FF G)^n \longrightarrow \FF\quad \text{by}\quad \langle e_p,q \rangle = \delta_{p,q}.
\end{equation}

\smallskip

\begin{lemma}\label{lemma:omega-orthogonal}
Let $T^n=T\cap (\FF G)^n$ be the $n$-th component of $T$. Then $\Omega_n(G)$ is the orthogonal complement of $T^n$ with respect to the bilinear form \eqref{eq:bilinear form}. 
\end{lemma}

\begin{proof}
 First, we consider the following bilinear form
\begin{equation}\label{eq bilinear form 2}
 \RR_n(G)\times \RR^n(G) \to \FF, \quad B(e_p,e^q)=\delta_{p,q},
\end{equation}
where $\delta_{p,q}$ is the Kroneck delta.
With respect to~\eqref{eq bilinear form 2}, one can check that~$\NN^n(G)$ and $\partial (\NN^{n-1}(G))$ are the orthogonal complements of $\AA_n(G)$ and~$\partial^{-1}(\AA_{n-1}(G))$, respectively.
Then it follows from~\eqref{eq defn Omega n} that the orthogonal complement of $\Omega_n(G)$ with respect to the bilinear form~\eqref{eq bilinear form 2} is 
\[
\JJ^n(G)=\NN^n(G)+\partial (\NN^{n-1}(G)).
\]

Second, the bilinear form~\eqref{eq:bilinear form} is induced from~\eqref{eq bilinear form 2} by restricting~$\RR_n(G)$ to~$\AA_n(G)$ and the identifications
\[
\Hom_\FF(\AA_n,\FF)\cong \RR^n(G)/\NN^n(G)\xrightarrow{\varphi} (\FF G)^n,
\]
where $\varphi$ is an isomorphism given in~\eqref{eq:iso-R/N-KG}. Hence the orthogonal complement of~$\Omega_n(G)$ with respect to~\eqref{eq:bilinear form} is the image of $\JJ^n(G)/\NN^n(G)$ under $\varphi$, which is~$T^n$ by Lemma~\ref{lemma isomorphism}.
\end{proof}

\begin{corollary}\label{corollary:duality}
There is a non-degenerate bilinear form
\[
\langle -,- \rangle : \Omega_n(G) \times (\FF G/T)^n \longrightarrow \FF
\]
defined by 
$\langle \sum_p \alpha_p\cdot e_p , q +T^n  \rangle = \sum_p \alpha_p\cdot \delta_{p,q},$ where $\delta_{p,q}$ is the Kronecker delta. 
\end{corollary}

\subsection{Rational and integral path (co)homology} 

\begin{proposition}\label{prop:rationalisation}
For any finite graph $G$, there are isomorphisms 
\begin{equation}\label{eq:rationalisation}
\Omega_\bullet(G;\QQ) \cong \Omega_\bullet(G;\ZZ) \otimes \QQ \hspace{1cm} \PH_*(G;\QQ)\cong \PH_*(G;\ZZ)\otimes \QQ;
\end{equation}
\begin{equation}
\Omega^\bullet(G;\QQ) \cong 
\Omega^\bullet(G;\ZZ) \otimes \QQ, 
\hspace{1cm} 
\PH^*(G;\QQ)\cong  \PH^*(G;\ZZ)\otimes  \QQ.
\end{equation}
\end{proposition}
\begin{proof}
It is obvious that there are isomorphisms
\begin{equation}\label{eq isomorphism rational and integral coefficients}
\RR_\bullet(G;\QQ)\cong \RR_\bullet(G;\ZZ)\otimes \QQ
\text{ and } \AA_*(G;\QQ)\cong \AA_*(G;\ZZ)\otimes \QQ\, .
\end{equation}
 For any commutative ring $\KK$, there is an exact sequence 
\begin{equation}\label{eq ses}
  0 \longrightarrow \Omega_n(G;\KK) \longrightarrow \RR_n(G;\KK) \overset{f}\longrightarrow \frac{\RR_n(G;\KK)}{\AA_n(G;\KK)}  \oplus \frac{\RR_{n-1}(G;\KK)}{\AA_{n-1}(G;\KK)},  
\end{equation}
where $f(x)=(x+\AA_n,\partial(x)+\AA_{n-1})$ (see \cite[Lemma 3.4]{ivanov2024simplicial}), and the map
\[\Omega_n(G;\KK) \to \RR_n(G;\KK)\]
is an embedding. 
Since $-\otimes \QQ$ is an exact functor, we obtain a diagram of exact sequences
\[
\begin{tikzcd}
 0 \ar[r]& \Omega_n(G;\ZZ)\otimes \QQ \ar[r]\ar[d]& \RR_n(G;\ZZ)\otimes \QQ \ar[r,""]\ar[d,"\cong"] &\left(\frac{\RR_n(G;\ZZ)}{\AA_n(G;\ZZ)}  \oplus \frac{\RR_{n-1}(G;\ZZ)}{\AA_{n-1}(G;\ZZ)}\right)\otimes \QQ\ar[d,"\cong"]\\
 0 \ar[r]& \Omega_n(G;\QQ) \ar[r]& \RR_n(G;\QQ) \ar[r,""] &\frac{\RR_n(G;\QQ)}{\AA_n(G;\QQ)}  \oplus \frac{\RR_{n-1}(G;\QQ)}{\AA_{n-1}(G;\QQ)}   
\end{tikzcd}
\]
where the top row is obtained from \eqref{eq ses} for $\KK=\ZZ$ tensored with $\QQ$. 
The right two vertical maps are isomorphism by~\eqref{eq isomorphism rational and integral coefficients}, so is the left one.
 This means 
\begin{equation}\label{eq isomorphism rational and integral tensered}
\Omega_\bullet(G;\QQ)\cong \Omega_\bullet(G;\ZZ)\otimes \QQ.
\end{equation}
 Since $-\otimes \QQ$ is exact, it commutes with homology. Hence the isomorphism~\eqref{eq isomorphism rational and integral tensered} implies \[\PH_*(G;\QQ)\cong \PH_*(G;\ZZ)\otimes \QQ.\]

The argument for the dual version is similar. It is based on the usage of the exact sequence 
\begin{equation}
\NN^n(G;\KK) \oplus \NN^{n-1}(G;\KK) \overset{g}\longrightarrow \RR^n(G;\KK) \longrightarrow \Omega^n(G;\KK) \longrightarrow 0,
\end{equation}
where $g(x,y)=x+\partial(y).$
\end{proof}

\section{Graphs of short moves}\label{sec graphs of short moves}

To describe a basis for the path (co)chain complex, we introduce the concept of short moves between two different $n$-paths in a digraph $G$. This is central to defining the graph of short moves and its properties, which will be detailed next.

\subsection{Short moves}
For $n\geq 0$,
we say that two different $n$-paths 
\[
\text{$p=(p_0,\ldots,p_n)$ and $q=(q_0,\ldots,q_n)$ }
\]
in $G$ differ by a short move, if  
there is $1\leq i\leq n-1$ such that $p_j=q_j$ for each~$j\neq i$ and 
$d(p_{i-1},p_{i+1})=2$. In particular, if $n=0,1$ there are no short moves between $n$-paths. 

\begin{equation}
\begin{tikzcd}
& & & p_{i}\ar[rd] & & & \\
p_0 \ar[r]  & \cdots  \ar[r] & p_{i-1}\ar[ru]\ar[dr]  &   & p_{i+1}\ar[r] & \cdots \ar[r] & p_n \\
& & & q_{i}\ar[ru] & & & 
\end{tikzcd}
\end{equation}
Note that the equality $d(p_{i-1},p_{i+1})=2$ means that there is no  arrow from $p_{i-1}$ to~$p_{i+1}$ and $p_{i-1}\neq p_{i+1}.$ If we need to specify the index $i,$ we say that it is a short move of type $i.$

\begin{definition}\label{def graph of short moves}
For $n\geq 0$, the graph of short moves of $G$, denoted by~$\SS_n=\SS_n(G)$, is an undirected simple graph whose vertices are $n$-paths of $G$, and there is an edge between two distinct $n$-paths if they differ by a short move.
Connected components of $\SS_n$ are called $\SS_n$-classes. We also label the edges of~$\SS_n$ with $n-1$ colors.  An edge has color~$i$ if it corresponds to a short move of type~$i$.
\end{definition}

Note that $\SS_0$ is a discrete graph consisting of the vertices of~$G$, and~$\SS_1$ is a discrete graph whose vertices correspond to arrows in $G$.

\begin{remark} In 
\cite{ivanov2024diagonal}, the authors introduced a more general notion of $\ell$-short move for any~$\ell\geq 2,$ such that the short moves defined above are $2$-short moves. One can similarly define a graph $\SS_{n,\ell}$ using $\ell$-short moves for $\ell\geq 2.$ The $\ell$-short moves lead to an equivalence relation on paths, called $\ell$-short congruence. This congruence relation was used in \cite{ivanov2024diagonal} for the magnitude homology theory. Specifically, it was used to describe digraphs such that $\MH_{2,k}(G;\ZZ)=0$ for any $k>\ell.$ 
\end{remark}

\begin{example}\label{example:grid}
Let $G$ be one of the two following digraphs:
\begin{equation}
\begin{tikzcd}
0 \ar[r] \ar[d] & 1\ar[r] \ar[d] & 2 \ar[d] \\
2\ar[r] & 4\ar[r]  & 5
\end{tikzcd}
\hspace{1cm}
\begin{tikzcd}
0 \ar[r] \ar[rr,bend left = 10mm] \ar[d] & 1\ar[r] \ar[d] & 2 \ar[d] \\
2\ar[r] \ar[rr,bend right = 10mm] & 4\ar[r]  & 5
\end{tikzcd}
\end{equation}
Then $\SS_3$ can be described as follows.
\begin{equation}
\begin{tikzcd}
0125 \ar[r,-,"2"] & 0145 \ar[r,-,"1"] & 0245
\end{tikzcd}
\end{equation}
\end{example}
\begin{example}\label{example:general_grid} 
For both of the digraphs 
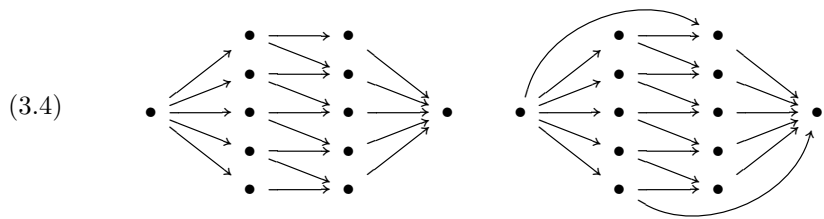
\begin{equation}
\begin{tikzcd}[row sep=3pt]
& \bullet \ar[r] \ar[rd]  & \bullet \ar[rdd] & \\
&    \bullet \ar[r] \ar[rd] & \bullet \ar[rd]  & \\
\bullet  \ar[r] \ar[ru] \ar[rd] \ar[ruu] \ar[rdd] & \bullet \ar[r] \ar[rd] & \bullet \ar[r]  & \bullet  
\\
  & \bullet\ar[r] \ar[rd]  & \bullet\ar[ru] &    \\  
& \bullet \ar[r]  & \bullet \ar[ruu]  & 
\end{tikzcd}   
\hspace{5mm}
\begin{tikzcd}[row sep=3pt]
& \bullet \ar[r] \ar[rd]  & \bullet \ar[rdd] & \\
&    \bullet \ar[r] \ar[rd] & \bullet \ar[rd]  & \\
\bullet \ar[rruu,bend left = 18mm] \ar[r] \ar[ru] \ar[rd] \ar[ruu] \ar[rdd] & \bullet \ar[r] \ar[rd] & \bullet \ar[r]  & \bullet  
\\
  & \bullet\ar[r] \ar[rd]  & \bullet\ar[ru] &    \\  
& \bullet \ar[r] \ar[rruu,bend right = 18mm] & \bullet \ar[ruu]  & 
\end{tikzcd}    
\end{equation}
$\SS_3$ is a line graph with an odd number of vertices
\begin{equation}
\begin{tikzcd}
\bullet \ar[r,-,"1"] & \bullet \ar[r,-,"2"] & \bullet  \ar[r,white,"\color{black}{\dots}"] & \bullet  \ar[r,-,"2"]  &  \bullet
\end{tikzcd}
\end{equation}
\end{example}

\begin{example}\label{example:cube}
Let $G$ be a directed $3$-cube.
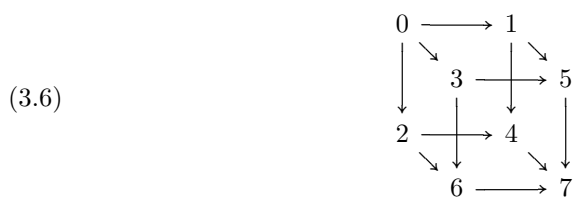
\begin{equation}
\begin{tikzcd}[row sep=7pt, column sep=7pt]
0\ar[rr]\ar[dd] \ar[rd] & & 1\ar[dd] \ar[dr] \\
  & 3 \ar[rr] \ar[dd] & & 5\ar[dd] \\
2\ar[rr] \ar[rd] & & 4\ar[rd]  \\
  & 6 \ar[rr] & &  7 
\end{tikzcd}
\end{equation}
Then $\SS_3$ is a hexagon. 
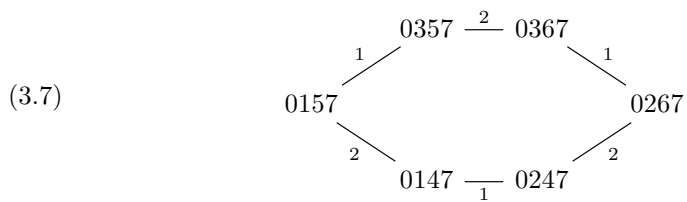
\begin{equation}
\begin{tikzcd}[row sep=15pt, column sep=15pt]
&   0357 \ar[r,-,"2"] & 0367  \ar[rd,-,"1"]  & \\
0157\ar[ru,-,"1"]  &&& 0267 \ar[ld,-,"2"] \\
& 0147 \ar[lu,-,"2"]  & 0247 \ar[l,-,"1"] &
\end{tikzcd}
\end{equation}
\end{example}

\begin{example}\label{example:trapezohedron}
Let $G$ be a ``graph of trapezohedron''. 
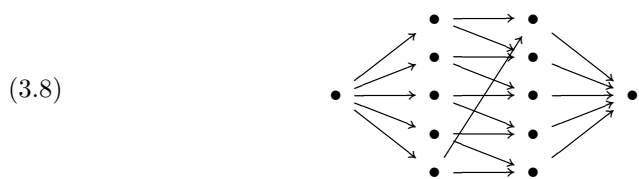
\begin{equation}
\begin{tikzcd}[row sep=3pt]
& \bullet \ar[r] \ar[rd]  & \bullet \ar[rdd] & \\
&    \bullet \ar[r] \ar[rd] & \bullet \ar[rd]  & \\
\bullet \ar[r] \ar[ru] \ar[rd] \ar[ruu] \ar[rdd] & \bullet \ar[r] \ar[rd] & \bullet \ar[r]  & \bullet  
\\
  & \bullet\ar[r] \ar[rd]  & \bullet\ar[ru] &    \\  
& \bullet \ar[r] \ar[ruuuu] & \bullet \ar[ruu]  & 
\end{tikzcd}    
\end{equation}
Then $\SS_3$ is an even cycle. 
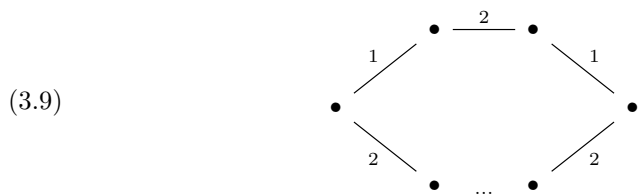
\begin{equation}
\begin{tikzcd}
&   \bullet \ar[r,-,"2"] & \bullet  \ar[rd,-,"1"]  & \\
\bullet \ar[ru,-,"1"]  &&& \bullet \ar[ld,-,"2"] \\
& \bullet \ar[lu,-,"2"]  & \bullet \ar[l,white,"\textcolor{black}{\dots}"] &
\end{tikzcd}
\end{equation}
\end{example}

\subsection{Thin, thick and bipartite \texorpdfstring{$\SS_n$}{}-classes}

We say that $G$ is a digraph without multisquares if for any two vertices $x,y$ such that $d(x,y)=2$, there are at most two distinct $2$-paths from $x$ to $y$.

\begin{definition}\label{def thin and thick}
Let $G$ be a digraph without multisquares.
\begin{enumerate}
\item A pair $(x,y)$ with $d(x,y)=2$ is called \emph{thin}  if there is only one $2$-path from~$x$ to $y$, and \emph{thick} if there are two different $2$-paths from $x$ to~$y$.
\item An $n$-path $(p_0,p_1,\ldots, p_n)$ is called \emph{thin} if there exists $0<i<n$ such that the pair~$(p_{i-1},p_{i+1})$ is thin. An $\SS_n$-class is called \emph{thin} if it contains at least one thin path. An $\SS_n$-class is \emph{thick} if it is not thin. 
 \end{enumerate}
\end{definition}
In other words, an $\SS_n$-class is thick if all its elements are not thin paths. 
In particular, every $\SS_n$-class is thick if $n=0$ or $1$.
In Examples~\ref{example:grid} and   \ref{example:general_grid}, the~$\SS_3$-class is thin for the left-hand digraph and thick for the right-hand digraph. 
In Examples~\ref{example:cube} and~\ref{example:trapezohedron} the $\SS_3$-classes are thick. 

\begin{definition}\label{def bipartite Sn class}
An $\SS_n$-class is called \emph{bipartite} if it is bipartite as a graph. 
\end{definition}
In Examples~\ref{example:grid}, \ref{example:general_grid}, \ref{example:cube} and~\ref{example:trapezohedron}, the $\SS_3$-classes are bipartite.  In Subsection~\ref{sec main example}, we will give an  example that has a non-bipartite $\SS_4$-class
due to the existence of a~9-cycle in this $\SS_4$-class.

\begin{proposition} \label{prop:S^3-bipartite}
Let $G$ be a digraph without multisquares. Then any $\SS_n$-class is bipartite for $n\leq 3$. In other words, the graph $\SS_n$ is bipartite for $n\leq 3$.
\end{proposition}
\begin{proof} For $n\leq 2$ the statement is obvious. Let us prove it for $n=3.$ 
The graph $\SS_3$ has edge labelled by two colors $1$ and $2.$
Since $G$ has no multisquares, any vertex~$p$ of $\SS_3$ is adjacent to at most one edge of color $1$ and at most one edge of color $2.$ Therefore any vertex of  $\SS_3$  has degree at most $2.$ Hence $\SS_3$ is a union of cycles and line graphs. Moreover, the two edges have different colors for any vertex of $\SS_3$ with degree two. It follows that all the cycles are even. Therefore $\SS_3$ is bipartite. 
\end{proof}

\section{The basis of the path (co)chain complex}\label{section basis}

In this section, a digraph $G$ is assumed to be without multisquares unless stated otherwise.
We continue to let $\KK$ be a commutative ring and $\FF$ be a field.

\subsection{The basis of \texorpdfstring{$\Omega^n$}{}} 

The following lemma is used to describe the basis of~$\Omega^n(G)$.

\begin{lemma}\label{lemma:module_of_graph}
Let $\Gamma$ be an undirected simple graph, $\mathcal V=\KK\cdot V(\Gamma)$ be the free module generated by the set of vertices, and $\mathcal E \leq \mathcal V$ be its submodule generated by elements of the form $v+v',$ where $(v,v')\in E(\Gamma).$ Assume that $B\subseteq V(\Gamma)$ is a complete set of representatives of bipartite connected components of $\Gamma,$ and  $N\subseteq V(\Gamma)$ is a complete set of representatives of non-bipartite connected components of $\Gamma.$ 
Then the obvious maps $B,N\to \mathcal V/\mathcal E$ induce an isomorphism 
\begin{equation}
\mathcal V / \mathcal E \cong  (\KK\cdot B) \oplus ((\KK/2\KK)\cdot N ).
\end{equation}
\end{lemma}

\begin{proof}
We present $\Gamma=\coprod_i \Gamma_i$ as a disjoint union of its connected components. For each $\Gamma_i$ we obtain a free module $\mathcal{V}_i$ generated by $V(\Gamma_i)$ and a submodule $\mathcal{E}_i\subseteq \mathcal{V}_i$ generated by the sums $v+v'$ where $(v,v')\in E(\Gamma_i).$ It is easy to see that 
\[
\mathcal{V}=\bigoplus_i \mathcal{V}_i, \hspace{5mm}  \mathcal{E}=\bigoplus_i \mathcal{E}_i, \hspace{5mm} \mathcal{V}/\mathcal{E}\cong \bigoplus_i \mathcal{V}_i/\mathcal{E}_i.
\]
Therefore, it suffices to prove the statement for a connected graph $\Gamma$.

Assume that $\Gamma$ is a connected bipartite graph. We present the set of its vertices as the disjoint union  $V(\Gamma)=U\sqcup W$ such that there are no edges between vertices of~$U,$ and there are no edges between vertices of $W.$ It is easy to see that in this case for any the module $\mathcal E$ is generated by elements of the form $u-u', w-w', u+w $ where~$u,u'\in U, w,w'\in W.$ Then we can consider a homomorphism $f: \mathcal{V}\to \KK$ sending all elements of~$U$ to $1,$ all elements of $W$ to $-1$ and prove that $\mathcal{E}=\Ker(f).$ Therefore $\mathcal{V}/\mathcal{E}\cong \KK.$

Assume that $\Gamma$ is connected non-bipartite graph. Then there is an odd cycle: a sequence of vertices $v_0,\dots,v_{2k+1}$ such that $(v_i,v_{i+1})\in E(\Gamma)$ and $v_{2k+1}=v_0.$ Summing up the elements $(-1)^i(v_i+v_{i+1})\in \mathcal{E},$ we obtain $2v_0\in \mathcal{E}.$ Since $\Gamma$ is connected, this implies that $2v\in \mathcal{E}$ for any vertex $v.$ Therefore $\mathcal{E}$ is generated by all elements of the form $2v$ and $v-v',$ where $(v,v')\in E(\Gamma).$ Consider a homomorphism~$f : \mathcal V\to \KK/2\KK$ sending $v$ to $1\in \KK/2\KK.$ It is easy to see that $\Ker(f)=\mathcal{E}.$ Therefore $\mathcal{V}/\mathcal{E}\cong \KK/2\KK.$
\end{proof}

We write the ideal $T \triangleleft \KK G$ as the sum of two ideals 
\begin{equation}
T=T_{\sf thick}+T_{\sf thin},
\end{equation}
where $T_{\sf thick}$ is generated by elements $t_{x,y}$ for thick pairs $(x,y),$ and $T_{\sf thin}$ is generated by elements $t_{x,y}$ for thin pairs $(x,y).$ 

\begin{lemma} \label{lemma:T-generators} \ 

\begin{enumerate}
    \item The $n$-th component $T_{\sf thick}^n$ is spanned by elements in the form $q+q',$ where $(q,q')$ is an edge of $\SS_n;$
    \item The $n$-th component $T_{\sf thin}^n$ is spanned by thin $n$-paths.
\end{enumerate}
\end{lemma}
\begin{proof}
The assertions follow immediately from the definitions of thin and thick pairs.
\end{proof}

In what follows, we identify $\Omega^\bullet(G;\KK)$ and $\KK G/T.$

\begin{theorem}\label{thm basis of path cochain complex}
Let $\KK$ be a commutative ring, and $G$ be a digraph without multisquares.
Let $B$ be a complete set of representatives of thick bipartite $\SS_n$-classes, and $N$ be a complete set of representatives of thick non-bipartite $\SS_n$-classes. Then the obvious maps $\KK\cdot B,~ \KK\cdot N \to \Omega^n(G;\KK)$ induce an isomorphism of~$\KK$-modules
\begin{equation}
\Omega^n(G;\KK) \cong (\KK\cdot B) \oplus ((\KK/2\KK)\cdot N).
\end{equation}
Moreover, thin $\SS_n$-classes map to zero in $\Omega^n(G;\KK).$ 
\end{theorem}
\begin{proof} 
By Lemmas \ref{lemma:module_of_graph} and \ref{lemma:T-generators}\,(1), we obtain an isomorphism
\[(\KK G/T_{\sf thick})^n\cong (\KK\cdot B')\oplus \left((\KK/2\KK)\cdot N'\right),\]
where $B'$ is a complete set of representatives of bipartite $\SS_n$-classes and $N'$ is a complete set of representatives of non-bipartite $\SS_n$-classes. The image of the map
\[T^n_{\sf thin}\longrightarrow (\KK G/T_{\sf thick})^n\]
is spanned by representatives of thin $\SS_n$-classes, from which the assertion follows.
\end{proof}

We apply Theorem~\ref{thm basis of path cochain complex} to the case where $\KK=\FF$ is a field.

\begin{corollary}\label{corollary:basis:omega^n}\ 

\begin{enumerate}
    \item If ${\sf char}(\FF)=2,$ then a complete set of representatives of thick $\SS_n$-classes is a basis of  $\Omega^n(G;\FF).$
    \item If ${\sf char}(\FF)\neq 2,$ then a complete set of representatives of thick bipartite $\SS_n$-classes is a basis of $\Omega^n(G;\FF).$ 
\end{enumerate}
\end{corollary}

\begin{corollary}\label{cor:dim:n<4}
For any two fields $\FF$ and~$\FF'$, we have
\begin{equation}
\dim_\FF(\Omega^n(G;\FF))= \dim_{\FF'}( \Omega^n(G;\FF')), \hspace{5mm} \text{ for }  \hspace{5mm} 0\leq n\leq 3.
\end{equation}
\end{corollary}
\begin{proof}
For $n=0,1$ it is obvious. For $n=2$, it is known that the dimension of~$\Omega^2(G;\FF)$ is the number of triangles plus the number of squares. For $n=3$ it follows from the fact that there are no non-bipartite short $\SS_3$-classes (Proposition~\ref{prop:S^3-bipartite}).
\end{proof}

\begin{corollary}\label{cor:not-equal}
Suppose that $G$ has a thick non-bipartite $\SS_n$-class. Then~$\Omega^n(G;\ZZ)$ has $2$-torsion and 
\[
   \dim_{\mathbb{Q}}(\Omega^n(G;\mathbb{Q})) < \dim_{\ZZ/2\ZZ}(\Omega^n(G;\ZZ/2\ZZ)). 
\]
\end{corollary}

\smallskip

\subsection{The dual basis of \texorpdfstring{$\Omega_n$}{}}

Now we work over a field $\FF$. Again assume $G$ has no multisquares.
For each thick bipartite $\SS_n$-class $c,$ we choose a partition 
\[c=c^+ \sqcup c^-\]
such that there are no short moves between elements inside $c^+,$ and between elements inside $c^-.$ Then we set 
\begin{equation}
s_c =\sum_{p\in c^+} e_p - \sum_{p\in c^-} e_p. 
\end{equation}
If ${\sf char}(\FF)=2,$ for each thick $\SS_n$-class $c$ (not necessarily bipartite), we also set
\begin{equation}
s'_c = \sum_{p\in c} e_p.
\end{equation}

\begin{theorem}\label{theorem:dual-basis}
The following statements hold:
\begin{enumerate}
\item If ${\sf char}(\FF)=2,$ then the elements of the form $s'_c$ form a basis of $\Omega_n(G),$ where $c$ runs over all thick $\SS_n$-classes. 
\item If ${\sf char}(\FF)\neq 2,$ then the elements of the form $s_c$ form a basis of $\Omega_n(G),$ where $c$ runs over all thick bipartite $\SS_n$-classes.
\end{enumerate}
Moreover, these bases are dual to bases described in Corollary \ref{corollary:basis:omega^n} with respect to the non-degenerate bilinear form described in Corollary \ref{corollary:duality}.  
\end{theorem}
\begin{proof}
Let us prove part (2); the proof of part~(1) is similar. First of all, we need to check that~$s_c\in \Omega_n(G)$ for any thick bipartite $\SS_n$-class $c$. By Lemmas~\ref{lemma:omega-orthogonal} and~\ref{lemma:T-generators}, it suffices to check that $s_c$ is orthogonal to  $p+q$ for $(p,q)\in E(\SS_n)$ and orthogonal to any thin $n$-path $p.$ If $p$ is a thin $n$-path, then $p\notin c,$ so $s_c$ is obviously orthogonal to $p.$ If $(p,q)\in E(\SS_n),$ then they are in the same $\SS_n$-class. 
If $p,q\notin c,$ again, it is obvious that $s_c$ is orthogonal to $p+q.$ Assume that $p,q\in c.$ Since $(p,q)\in E(\SS_n),$ we obtain either $p\in c^+$ and $q\in c^-,$ or $p\in c^-$ and $q\in c^+.$ Without lost of generality we assume that $p\in c^+$ and $q\in c^-.$ Then  we have
\[\begin{split}
\left\langle \sum_{r\in c^+}e_r, p \right\rangle =1,\quad & \left\langle \sum_{r\in c^-}e_r,p \right\rangle =0;\\
\left\langle \sum_{r\in c^-}e_r,q \right\rangle =1,\quad & \left\langle \sum_{r\in c^+}e_r,q \right\rangle =0.
\end{split}\]
Therefore $\langle s_c ,p+q \rangle=0.$

Consider a complete set $B\subseteq \Omega^n(G)$ of representatives of thick bipartite $\SS_n$-classes,  which is a basis of $\Omega^n(G)$ by Corollary \ref{corollary:basis:omega^n}. For any $b\in B$ we have
\[
\langle s_c,b \rangle=\begin{cases}
    1 & \text{ if } b\in c;\\
    0 & \text{ if } b\notin c.
\end{cases}
\]Therefore the elements $s_c$ form the dual basis of $B$ in $\Omega_n(G).$
\end{proof}

\begin{remark}
For $n=3$ our basis of $\Omega_3(G;\FF)$ coincides with the basis of trapezohedral paths and their merging images constructed by  Grigor’yan in \cite[Theorem~2.10]{grigor2022advances}. This follows from the proof of \cite[Theorem 2.10]{grigor2022advances}, in which the graph $P$ considered coincides with the connected component of $\SS_3(G).$
\end{remark}

\section{Nonexistence of spaces}

In this section we introduce the notion of the path Euler characteristic of a digraph and demonstrate its use in proving the nonexistence of spaces whose homology coincides with the path homology of certain digraphs with coefficients in different commutative rings simultaneously.  

\subsection{Euler characteristic}
We say that a space $X$ has homology of finite type if the integral homology group $H_n(X;\ZZ)$ is finitely generated for any $n.$ Similarly, a chain complex $C_\bullet$ of free abelian groups has homology of finite type if $H_n(C_\bullet)$ is finitely generated for any~$n.$ 

Given a chain complex $C_\bullet$ of free abelian groups and a field $\FF,$ we say that the Euler characteristic $\chi^\FF(C_\bullet)$ of $C_\bullet$ over $\FF$ is defined, if~$H_n(C_\bullet\otimes_\ZZ \FF)$ is finitely dimensional for all $n,$ and equals to zero for all but a finite number of $n$'s. In this case, we set 
\begin{equation}
\chi^\FF(C_\bullet) = \sum_{n}(-1)^n \dim_\FF( H_n(C_\bullet\otimes_\ZZ \FF)).
\end{equation}

The following lemma seems to be known but we add it here for completeness. 
\begin{lemma}\label{lemma:complexes}
Let $C_\bullet$ be a chain complex of free abelian groups with homology of finite type and $p$ be a prime. 
Suppose $\chi^\QQ(C_\bullet)$ and $\chi^{\ZZ/p\ZZ}(C_\bullet)$ are defined. Then 
\[
\chi^\QQ(C_\bullet) = \chi^{\ZZ/p\ZZ}(C_\bullet).
\]
\end{lemma}
\begin{proof}
Since $C_\bullet$ has homology of finite type, we obtain \[H_n(C_\bullet)\cong \ZZ^{a_n} \oplus (\ZZ/p\ZZ)^{b_n} \oplus A_n,\]
where $A_n$ is a finite abelian group that does not have $p$-torsion. By the universal coefficient theorem for chain complexes \cite[Th.3.6.2]{weibel1995introduction}, we have
\[
H_n(C_\bullet\otimes_\ZZ \QQ)\cong \QQ^{a_n} \text{ and }
H_n(C_\bullet\otimes_\ZZ \ZZ/p\ZZ)\cong (\ZZ/p\ZZ)^{a_n+b_n+b_{n-1}}.    
\]

The hypothesis that $\chi^\QQ(C_\bullet)$ and $\chi^{\ZZ/p}(C_\bullet)$ are defined implies that $a_n$ and $b_n$ are zero for all but a finite number of $n$'s. Therefore, we have
\[\chi^\QQ(C_\bullet)=\sum_n(-1)^n a_n = \sum_n (-1)^n (a_n+b_n+b_{n-1}) = \chi^{\ZZ/p\ZZ}(C_\bullet).\qedhere
\]
\end{proof}

In the next example, we see that the condition of having homology of finite type is sufficient for the statement of Lemma~\ref{lemma:complexes}.
\begin{example}
Let $C_\bullet$ be a free resolution of the abelian group $\QQ/\ZZ$ (a chain complex of free abelian groups such that $C_n=0$ for $n<0$, whose homology is concentrated in degree $0,$ and isomorphic to $\QQ/\ZZ$).  
Then~$\chi^\QQ(C_\bullet)$ and~$\chi^{\ZZ/p\ZZ}(C_\bullet)$ are defined but
\begin{equation}
   0=\chi^\QQ(C_\bullet) \neq \chi^{\ZZ/p\ZZ}(C_\bullet)=-1. 
\end{equation}
This is because the homology of $C_\bullet\otimes_\ZZ \QQ$ is trivial, and the homology of  $C_\bullet\otimes_\ZZ \ZZ/p\ZZ$ is concentrated in degree one: $H_1(C_\bullet\otimes_\ZZ \ZZ/p\ZZ) \cong \Tor(\QQ/\ZZ,\ZZ/p\ZZ) \cong \ZZ/p\ZZ.$ 
\end{example}

\subsection{Path Euler characteristic}

Let $\FF$ be a field. The path Euler characteristic~$\chi^\FF(G)$ of a finite digraph $G$ (may contain multisquares) is defined, if the path homology groups~$\PH_n(G;\FF)$ are trivial for all but a finite number of~$n$'s. 
In this case, we set
\begin{equation}
\chi^{\FF}(G) = \sum_{n} (-1)^n \dim_{\FF}(\PH_n(G;\FF)).
\end{equation}
If there exists $m$ such that  $\Omega_m(G;\FF)=0,$ then  $\Omega_n(G;\FF)=0$ for $n\geq m.$ In this case,
\begin{equation}
\chi^{\FF}(G) = \sum_{n=0}^{m-1} (-1)^n \dim_{\FF}(\Omega_n(G;\FF)).
\end{equation}

In the case that $G$ has no multisquares, we apply Corollary~\ref{corollary:basis:omega^n} to compute the path Euler characteristics. Denote by $N_n$ the number of thick bipartite  $\SS_n$-classes and by $N'_n$ the number of thick $\SS_n$-classes. Then we have 
\begin{equation}
\chi^\FF(G) 
= 
\begin{cases}
\sum_n (-1)^n N_n, & \text{ if } {\sf char}(\FF)\neq 2\ \text{ and }\   N_m=0 \text{ for some } m; \\
\sum_n (-1)^n N'_n, &  \text{ if } {\sf char}(\FF)= 2\ \text{ and }\   N'_m=0 \text{ for some } m.
\end{cases}
\end{equation}

\subsection{Path homology and nonexistence of spaces}

\begin{proposition}\label{prop:general:non-existance}
Let $G$ be a finite digraph and $p$ be a prime such that the path Euler characteristics $\chi^\QQ(G)$ and $\chi^{\ZZ/p\ZZ}(G)$ are defined. Suppose $\chi^\QQ(G) \neq \chi^{\ZZ/p\ZZ}(G)$.
Then the following statements hold:
\begin{enumerate}

\item 
There is no space $X$ such that there are isomorphisms of graded abelian groups 
\[\text{$\PH_*(G;\KK)\cong H_*(X;\KK)$ \ \  for \ \  $\KK=\ZZ$ \ \  and \ \ $\KK=\ZZ/p\ZZ$.}
\]

\item There is no space $X$ with homology of finite type such that there are isomorphisms of vector spaces
\[
\text{$\PH_*(G;\FF)\cong H_*(X;\FF)$ \ \  for \ \  $\FF=\QQ$ \ and \ $\FF=\ZZ/p\ZZ$.}
\]

\item There is no chain complex $C_\bullet$ of free abelian groups with homology of finite type such that 
\[\text{$\PH_*(G;\FF) \cong H_*(C_\bullet \otimes_\ZZ \FF)$ \ \  for \ \ $\FF=\QQ$ \ and \ $\FF=\ZZ/p\ZZ$.}\]

\item There exists $n$ such that there is no ``universal coefficient'' short exact sequence 
\[ 0 \to \PH_n(G;\ZZ)\otimes \ZZ/p\ZZ \to  \PH_n(G;\ZZ/p\ZZ) \to  \Tor(\PH_{n-1}(G;\ZZ),\ZZ/p\ZZ) \to 0.\]
 \end{enumerate}
\end{proposition}
\begin{proof}

We first show part (3), while part (2) follows from (3). Assume the contrary, that $C_\bullet$ is a chain complex of free abelian groups of finite homology type such that 
\[\text{$H_n(C_\bullet\otimes_\ZZ \FF)\cong \PH_\bullet(G;\FF)$ for~$\FF=\QQ$ and~$\FF=\ZZ/p\ZZ.$}\]
Then $\chi^\QQ(C_\bullet)$ and $\chi^{\ZZ/p\ZZ}(C_\bullet)$ are defined. Lemma~\ref{lemma:complexes} implies $\chi^\QQ(C_\bullet)=\chi^{\ZZ/p\ZZ}(C_\bullet)$,
so $\chi^\QQ(G)= \chi^{\ZZ/p\ZZ}(G)$ by assumption. 
This makes a contradiction. 

Next we prove part (1). Suppose there is a space $X$ such that 
\[\PH_*(G;\ZZ)\cong H_*(X;\ZZ).\]
As $G$ is finite, the space $X$ has finite homology type. By Proposition \ref{prop:rationalisation}, we have 
\[\PH_*(G;\QQ) \cong  \PH_*(G;\ZZ) \otimes \QQ \cong H_*(X;\ZZ) \otimes \QQ \cong H_*(X;\QQ).\]
This contradicts to part (2).

Finally, we prove part (4). 
We follow an argument similar to the proof of Lemma~\ref{lemma:complexes}. 
Assume the contrary. 
Since $G$ is finite, we have 
\[\PH_n(G;\ZZ)\cong \ZZ^{a_n} \oplus (\ZZ/p\ZZ)^{b_n} \oplus A_n,\] 
where $A_n$ is a finite abelian group without $p$-torsion. By Proposition \ref{prop:rationalisation}, we have~$\PH_n(G;\QQ)\cong \QQ^{a_n}.$
By assumption, we have 
$\PH_n(G;\ZZ/p\ZZ)\cong (\ZZ/p\ZZ)^{a_n+b_n+b_{n-1}}.$ 
The fact that $\chi^\QQ(G)$ and $\chi^{\ZZ/p\ZZ}(G)$ are defined implies that $a_n$ and $b_n$ are zero for all but a finite number of $n$'s. Therefore 
\[\chi^\QQ(G)=\sum_n(-1)^n a_n = \sum_n (-1)^n (a_n+b_n+b_{n-1}) = \chi^{\ZZ/p\ZZ}(G),\]
which makes a contradiction.  
\end{proof}

Next we will give an explicit example that satisfies the assumption of Proposition~\ref{prop:general:non-existance}.

\subsection{The main example} \label{sec main example}

Let us first construct a auxiliary example of a digraph~$\GG',$ and later we will construct the main example $\GG,$ which  will be a modification of $\GG'.$ 

The digraph $\mathcal{G}'$ has vertices $x_0,x_4$ and $x^i_j,$ where $i\in \ZZ/3\ZZ$ and $1\leq j\leq 3$, with arrows:
\[x_0\to x^i_1,~ x^i_3\to x_4,~ x^i_j\to x^i_{j+1}  \text{ and } x^i_j \to x^{i+1}_{j+1} \quad
\text{for } j=1,2 \text{ and } i\in \ZZ/3\ZZ.\]
We picture it below. 
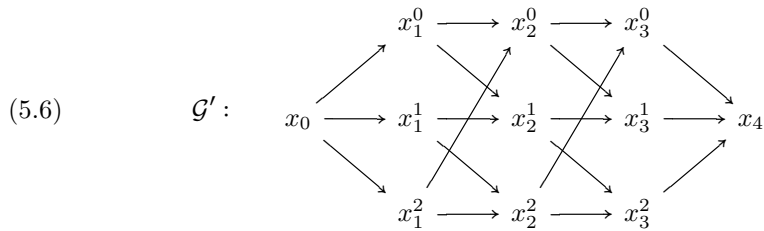
\begin{equation} 
\mathcal{G}': \hspace{5mm}
\begin{tikzcd}
  &  x_1^0 \ar[r] \ar[rd] & x_2^0 \ar[rd] \ar[r] & x_3^0 \ar[dr] &   \\
x_0\ar[r] \ar[ru] \ar[rd] & x_1^1 \ar[r] \ar[rd] & x_2^1 \ar[r] \ar[rd] & x_3^1\ar[r] & x_4 \\
  & x_1^2\ar[r] \ar[ruu] & x_2^2\ar[r]\ar[ruu] & x_3^2 \ar[ur]  &   
\end{tikzcd}    
\end{equation}

\

A $4$-path in $\GG'$ has the form $(x_0,x_1^{i_1},x_2^{i_2},x_3^{i_3},x_4),$ where $i_{k+1}\in \{i_{k},i_{k}+1\}.$ Therefore, there are $3\cdot 2\cdot 2 =12$ paths of length $4$. Let us denote them by $p_{i_1i_2i_3}=(x_0,x_1^{i_1},x_2^{i_2},x_3^{i_3},x_4).$ Then the digraph $\SS_4$ can be drawn as follows.

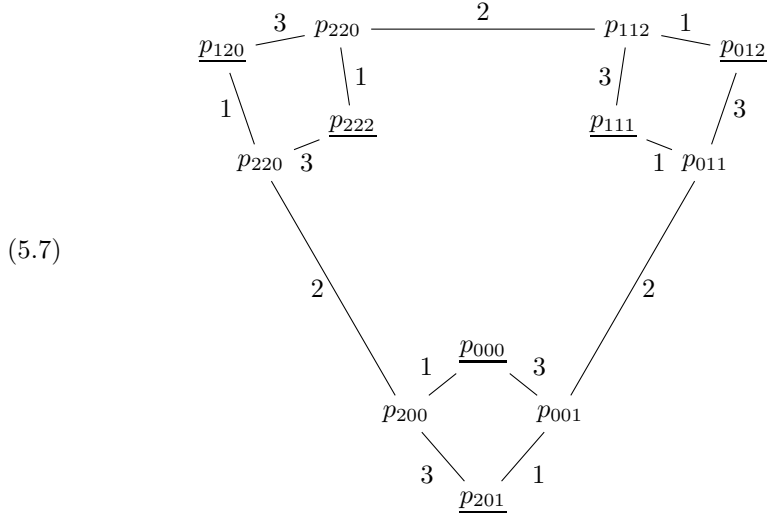
\begin{equation}\label{eq:S^4}
\begin{tikzpicture}[baseline=(current  bounding  box.center)]
\draw[line width=1pt]
({2*cos(0*360/3-90)},{2*sin(0*360/3-90)})  node (a1) {\underline{$p_{000}$}}

({2*cos(1*360/3-90)},{2*sin(1*360/3-90)})  node (a2) {\underline{$p_{111}$}}

({2*cos(2*360/3-90)},{2*sin(2*360/3-90)})  node (a3) {\underline{$p_{222}$}}

({4*cos(0*360/3-90)},{4*sin(0*360/3-90)})  node (b1) {\underline{$p_{201}$}}

({4*cos(1*360/3-90)},{4*sin(1*360/3-90)})  node (b2) {\underline{$p_{012}$}}

({4*cos(2*360/3-90)},{4*sin(2*360/3-90)})  node (b3) {$\underline{p_{120}}$}

({3*cos(0*360/3-90+20)},{3*sin(0*360/3-90+20)})  node (c1) {$p_{001}$}

({3*cos(1*360/3-90+20)},{3*sin(1*360/3-90+20)})  node (c2) {$p_{112}$}

({3*cos(2*360/3-90+20)},{3*sin(2*360/3-90+20)})  node (c3) {$p_{220}$}

({3*cos(0*360/3-90-20)},{3*sin(0*360/3-90-20)})  node (d1) {$p_{200}$}

({3*cos(1*360/3-90-20)},{3*sin(1*360/3-90-20)})  node (d2) {$p_{011}$}

({3*cos(2*360/3-90-20)},{3*sin(2*360/3-90-20)})  node (d3) {$p_{220}$}
;
\path[-]
(a1) edge  node[above right]{3}  (c1)
(c1) edge  node[below right]{1} (b1)
(b1) edge  node[below left]{3} (d1)
(d1) edge  node[above left]{1} (a1)
(a2) edge  node[left]{3} (c2)
(c2) edge  node[above]{1} (b2)
(b2) edge  node[right]{3} (d2)
(d2) edge  node[below]{1} (a2)
(a3) edge  node[below]{3} (c3)
(c3) edge  node[left]{1} (b3)
(b3) edge  node[above]{3} (d3)
(d3) edge  node[right]{1} (a3)
(c1) edge  node[right]{2} (d2)
(c2) edge  node[above]{2} (d3)
(c3) edge  node[left]{2} (d1)
;
\end{tikzpicture}
\end{equation} 
Note that this graph contains a $9$-cycle, and hence, the $\SS_4$-class is non-bipartite. 

It is easy to check that the only thin pairs of $\GG'$ are these six pairs 
\begin{equation}
(x_1^i,x_3^i), \hspace{5mm} (x_1^i,x_3^{i+2})\ \text{ for } \  i\in \ZZ/3\ZZ.   
\end{equation}
Therefore the underlined paths on the picture are the thin $4$-paths in $\GG'.$ So the~$\SS_4$-class of $\GG'$ is thin and vanishes in $\Omega^4(\GG',\ZZ/2\ZZ).$ Further we will add more arrows to the digraph so that the corresponding $\SS_4$-class will be thick.

The digraph $\GG$ is obtained from $\GG'$ by adding six additional arrows connecting thin pairs of $\GG',$ namely the pairs $(x_1^i,x_3^i)$ and $(x_1^i,x_3^{i+2})$  for $i\in \ZZ/3.$ 

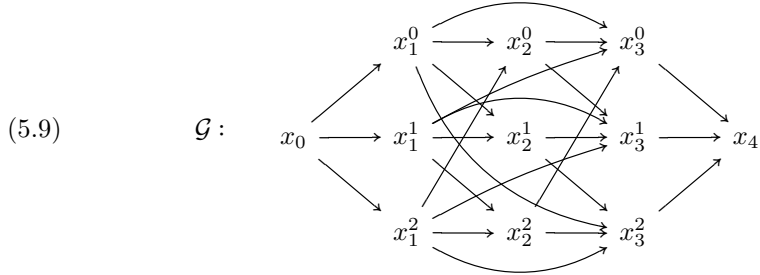
\begin{equation} \label{eq main example}
\GG: \hspace{5mm}
\begin{tikzcd}
  &  x_1^0 \ar[r] \ar[rr,
bend left = 10mm] \ar[rrdd,
bend right = 10mm] \ar[rd] & x_2^0 \ar[rd] \ar[r] & x_3^0 \ar[dr] &   \\
x_0\ar[r] \ar[ru] \ar[rd] & 
x_1^1 \ar[r] \ar[rd] \ar[rr,
bend left = 10mm] \ar[rru,
bend left =2mm]
& x_2^1 \ar[r] \ar[rd] & x_3^1\ar[r] & x_4 \\
  & x_1^2\ar[r] \ar[ruu] \ar[rr,
bend right=10mm] \ar[rru,
bend left=2mm] & x_2^2\ar[r]\ar[ruu] & x_3^2 \ar[ur]  &
\end{tikzcd}    
\end{equation}

\medskip

Note that $\SS_4(\GG)=\SS_4(\GG').$ However $(x_1^i,x_3^i)$ and $(x_1^i,x_3^{i+2})$ are not thin pairs of~$\GG,$ because $d(x_1^i,x_3^i)=d(x_1^i,x_3^{i+2})=1$ in $\GG.$

\begin{theorem}\label{theorem:example}
For the digraph $\GG$ defined in~\eqref{eq main example}, the following statements hold:
\begin{enumerate}
\item The rational and mod-$2$ path Euler characteristics of $\GG$ are not equal 
\[ \chi^\QQ(\GG) \neq \chi^{\ZZ/2\ZZ}(\GG).\]

\item $\Omega^4(\GG;\ZZ)$ has $2$-torsion.




\end{enumerate} 
\end{theorem}
\begin{proof}
We first check that $\GG$ has no multisquares and has no thin pairs. 
It is easy to check that the pairs $(x,y)$ such that $d(x,y)=2$ are the pairs in the form
\begin{equation} \label{eq pairs of distance 2}
\text{$(x_0,x_2^i), $ $(x_0,x_3^i),$ $(x_1^i,x_3^{i+1}),$ $(x_1^i,x_4),$ $(x_2^i,x_4).$}
\end{equation}
For pairs $(x,y)$ in~\eqref{eq pairs of distance 2}, there are exactly two vertices $v_1,v_2$ such that $(x,v_i,y)$ is a path:
\begin{align*}
(x_0,x_2^i): &\quad  x_1^i,~ x_1^{i-1};\\
(x_0,x_3^i): & \quad x_1^i,~ x_1^{i+1}; \\
(x_1^i,x_3^{i+1}):& \quad x_2^i,~ x_2^{i+1};\\
(x_1^i,x_4):& \quad x_3^i,~ x_3^{i+2};\\
(x_2^i,x_4):& \quad x_3^i,~ x_3^{i+1}. 
\end{align*}
Therefore $\GG$ has no multisquares, and has no thin paths.

The graph $\SS_4(\GG)$ coincides with the graph $\SS_4(G)$ drawn in \eqref{eq:S^4}. It follows that all $4$-paths of $\GG$ form a thick non-bipartite $\SS_4$-class. By Theorems \ref{theorem:dual-basis} and~\ref{theorem:omega:description} we have
\[
\Omega_4( \GG;\QQ)=0,  \hspace{5mm} \Omega_4( \GG;\ZZ/2\ZZ)\cong \ZZ/2\ZZ, \hspace{5mm} \Omega^4(\GG; \ZZ )\cong \ZZ/2\ZZ.
\]

Since there are no paths of length $\geq 5,$ we obtain $\Omega_n(G;\FF)=0$ for $n\geq 5$ and any field $\FF.$ On the other hand by Corollary \ref{cor:dim:n<4}, we have
\[
\text{$\dim_\QQ (\Omega_n(\GG;\QQ))=\dim_{\ZZ/2\ZZ} (\Omega_n(\GG;\ZZ/2\ZZ))$ for $n\leq 3.$}
\]It follows that 
\begin{equation}
 \chi^\QQ(\GG) = \chi^{\ZZ/2\ZZ}(\GG)-1.  \qedhere 
\end{equation}
\end{proof}

\smallskip

\begin{corollary}
The statements in Proposition \ref{prop:general:non-existance} hold for $\GG$. \qed
\end{corollary}

\printbibliography

\end{document}